\documentclass{amsart}

\usepackage{ amsthm, amssymb, amsmath, latexsym }
\usepackage{ array }
\usepackage{ booktabs }
\usepackage{ enumitem }
\usepackage{ mathdots }
\usepackage{ mathrsfs }
\usepackage{ url }
\usepackage{ color }

\theoremstyle{definition}

\newtheorem{definition}{Definition}[section]

\newtheorem{example}[definition]{Example}

\newtheorem{notation}[definition]{Notation}

\newtheorem{remark}[definition]{Remark}

\theoremstyle{plain}

\newtheorem{lemma}[definition]{Lemma}
\newtheorem{proposition}[definition]{Proposition}
\newtheorem{theorem}[definition]{Theorem}

\newcommand{\tass}{\mathcal{T\!A}}
\newcommand{\pass}{\mathcal{P\!A}}

\makeatletter
\def\l@section{\@tocline{1}{0pt}{1pc}{}{}}
\def\l@subsection{\@tocline{2}{0pt}{1pc}{4.6em}{}}
\def\l@subsubsection{\@tocline{3}{0pt}{1pc}{7.6em}{}}
\renewcommand{\tocsection}[3]{%
\indentlabel{\@ifnotempty{#2}{\makebox[1.25em][l]{\ignorespaces#1#2.}}}#3}
\renewcommand{\tocsubsection}[3]{%
\indentlabel{\@ifnotempty{#2}{\hspace*{1.25em}\makebox[2.00em][l]{\ignorespaces#1#2.}}}#3}
\renewcommand{\tocsubsubsection}[3]{%
\indentlabel{\@ifnotempty{#2}{\hspace*{3.25em}\makebox[2.75em][l]{\ignorespaces#1#2.}}}#3}
\makeatother

\allowdisplaybreaks

\begin{document}

\title[Operadic cohomology of associative $N$-tuple systems]
{Operadic approach to cohomology of\\associative triple and $N$-tuple systems}

\author{Fatemeh Bagherzadeh}

\address{Department of Mathematics and Statistics, University of Saskatchewan, Canada}

\email{bagherzadeh@math.usask.ca}

\author{Murray Bremner}

\address{Department of Mathematics and Statistics, University of Saskatchewan, Canada}

\email{bremner@math.usask.ca}




\thanks{The authors were supported by the Discovery Grant \emph{Algebraic Operads}
from NSERC, the Natural Sciences and Engineering Research Council of Canada.
The authors thank Vladimir Dotsenko and Martin Markl for helpful suggestions at 
various stages of this research.}

\begin{abstract}
The cup product in the cohomology of algebras over quadratic operads 
has been studied in the general setting of Koszul duality for operads. 
We study the cup product on the cohomology of $n$-ary totally associative 
algebras with an operation of even (homological) degree.
This cup product endows the cohomology with the structure of an $n$-ary 
partially associative algebra with an operation of even or odd degree
depending on the parity of $n$.
In the cases $n = 3$ and $n = 4$, we provide an explicit definition of 
this cup product and prove its basis properties.
\end{abstract}

\maketitle

\tableofcontents


\section{Introduction}


\subsection{Historical background}

During the 1940s, topologists discovered the homology theory of topological spaces
and some invariants of algebraic systems \cite{Weibel}.
It was proved that the cohomologies of a spherical space $Y$ and of a Lie group $G$ 
depend only on fundamental group $\pi_1(Y)$ and the corresponding Lie algebra $\mathfrak{g}$.

In 1945, Hochschild \cite{Hochschild} defined the cohomology of an associative algebra $A$ 
over a field $\mathbb{F}$ with coefficients in a bimodule $M$.
His cochain complex  $C^q(A;M)$ of degree $q$ is
the vector space of linear maps $A^{\otimes q} \rightarrow M$.
He showed that $H^1(A;M)$ vanishes
for every $M$ if and only if $A$ is a separable algebra.
Hochschild also showed that $H^2(A;M)$ is directly connected with extensions of $A$.
In particular if $A$ is semisimple then $H^2(A;M)$ vanishes for every $M$;
hence every nilpotent extension of $A$ splits.
Shukla \cite{Shukla} extended Hochschild's results to general coefficient rings.

In 1962, Gerstenhaber \cite{Gerstenhaber} showed that if $A$ is an associative ring 
then $H^*(A;A)$ is a graded ring with commutative cup product denoted $\smile$ and 
furthermore that $H^*(A;M)$ is an $H^*(A;A)$ bimodule.
He also defined a bracket $[-,-]$ on $H^*(A;A)$ that makes it a Lie ring;
the two products are related by the derivation law:
\[
[\alpha^m \smile \beta^n, \gamma^p]= [\alpha^m , \gamma^n] \smile \beta^p 
+ (-1)^{p(m-1)}\alpha^m \smile [\beta^n, \gamma^p]
\]
Graded rings with two operations satisfying these conditions are now called 
Gerstenhaber algebras and play a significant role in theoretical physics \cite{Getzler}.

In 1993, Deligne posed the conjecture that the Hochschild cohomology of an associative algebra
has the structure of an algebra over an operad over the (framed) little disks operad.
Following earlier work by Gerstenhaber \& Voronov \cite{GV} and many others, 
this conjecture was first proved by McClure \& Smith \cite{MS-D}.

In the early 1970s, Lister \cite{Lister}, based on earlier work by Hestenes \cite{Hestenes},
introduced ternary rings (associative triple systems):
additive subgroups of rings which are closed under the associative triple product.
He discussed embeddings of ternary rings in enveloping rings, 
representation theory, radicals, and semisimplicity.
In 1976, Carlsson \cite{Carlsson-1} made the first systematic study of their cohomology.
For an associative triple system $A$, she defined its cohomology in terms of the cohomology of 
its universal embedding.
A few years later \cite{Carlsson-2}, 
she introduced the notion of associative $n$-ary algebra, 
and proved analogues of the Whitehead lemmas and the Wedderburn structure theory.

Totally and partially associative operads with generator $\mu$ of arity $n$ and degree $d$ 
were introduced by Gnedbaye \cite{Gnedbaye}.
Low-dimensional cohomology for totally and partially associative triple systems
has been studied by Ataguema et al.~\cite{Ataguema}.
The Koszulity of these operads has been investigated by Markl \& Remm \cite{MR-1, MR-2}.

For a quadratic operad $\mathcal{P}$, the cohomology groups of a $\mathcal{P}$-algebra
were defined by Markl et al.~\cite{MSS}.
The main theorem connecting $\mathcal{P}$-algebras to algebras over its Koszul dual
$\mathcal{P}^!$ is that the cohomology of a $\mathcal{P}$-algebra has the structure of 
a $\mathcal{P}^!$-algebra \cite{Markl-2,LV}.
Loday \cite{Loday} studied the first explicit example of this correspondence:
for a Lie algebra $\mathfrak{g}$, he constructed a cup product on the Leibniz cohomology:
\[
\smile \colon
HL^p(\mathfrak{g}) \times HL^q(\mathfrak{g}) \longrightarrow HL^{p+q}(\mathfrak{g})
\]    
This product is neither commutative nor associative but satisfies the relation
\[
(f \smile g) \smile h= f \smile (g \smile h) 
+ (-1)^{|g||h|}f \smile (g \smile h),
\]
which makes $HL^*(\mathfrak{g})$ into a Zinbiel (dual Leibniz) algebra.


\subsection{Results of this paper}

In this paper we study the structure  of Koszul dual  of the totally associative nonsymmetric operads $\tass_d^n$ generated by one $n$-ary operation of (homological) degree $d$ through their operadic cohomology.
We focus on $n = 3$, $d = 0$, the operad $\tass_0^3$ is known to be Koszul by results of Markl \& Remm \cite[\S2.4]{MR-1}.
This operad generated by a ternary operation $x_1, x_2, x_3 \mapsto \mu(x_1,x_2,x_3)$
of degree zero that  satisfies these identical relations:
\[
\mu(\mu(x_1,x_2,x_3),x_4,x_5)=\mu(x_1,\mu(x_2,x_3,x_4),x_5)=\mu(x_1,x_2,\mu(x_3,x_4,x_5)).
\]
Its Koszul dual is the partially associative operad $\pass_1^3$ generated by a ternary operation
$x_1, x_2,x_3 \mapsto \mu^\ast(x_1,x_2,x_3)$ of (homological) degree one satisfiying this relation:
\[
\mu^\ast(\mu^\ast(x_1,x_2,x_3),x_4,x_5) + \mu^\ast(x_1,\mu^\ast(x_2,x_3,x_4),x_5) + \mu^\ast(x_1,x_2,\mu^\ast(x_3,x_4,x_5)) =0.
\]
We describe the operadic cochain complex used for calculation of the cohomology groups of a $\tass_0^3$-algebra $A$ with coefficients in $A$ in the sense of Loday \& Vallette \cite[\S8.12]{LV}.
Then we define a ternary cup product of degree one on this operadic cochain complex .
With this product the corresponding operadic cohomology $H^\bullet_{\tass_0^3}(A,A)$ has the structure of  $\pass_1^3$-algebra.


\section{Preliminaries}


\subsection{Nonsymmetric operads}

We work with nonsymmetric (ns) operads in the monoidal category 
of $\mathbb{Z}$-graded vector spaces with the Koszul sign rule.
Such an operad $\mathcal{P} = \{ \, \mathcal{P}(n) \mid n \ge 1 \, \}$ is a sequence
of vector spaces that satisfy the unity and associativity axioms
for partial compositions \cite{BD-AO,LV,MSS,Markl-1}.
The space $E$ of generating operations of $\mathcal{P}$ is bigraded by arity $a$ and degree $d$:
\[
E = \bigoplus_{a \ge 2} E(a),
\qquad\qquad
E(a) = \bigoplus_{d \in \mathbb{Z}} E(a)_d.
\]
Thus $E(a)$ is the vector space of operations of arity $a$ which is the direct sum
of vector spaces of operations of various (homological) degrees $d$.

We consider a single generating operation $\mu$ of arity $n \ge 2$
and degree 0.
Thus $E(a) = 0$ for $a \ne n$, $E(n)_d = 0$ for $d \ne 0$, and $E(n)_0 = \mathbb{F}\mu$.
Let $\Gamma(E)$ be the free operad generated by $E$, 
and $\Gamma(E)(N)$ its homogeneous subspace of arity $N$:
\[
\Gamma(E) = \bigoplus_{N \ge 1} \Gamma(E)(N).
\]
To start, $\Gamma(E)(1)$ has basis $\iota = x_1$ (identity unary operation), 
and $\Gamma(E)(n)$ has basis $\mu = \mu(x_1,\dots,x_n)$.
Every partial composition with $\mu$ replaces one argument by $n$ arguments,
increasing the arity by $n{-}1$,
so $\Gamma(E)(N) = 0$ unless $N \equiv 1$ (mod $n{-}1$).
In particular, $\Gamma(E)(2n{-}1)$ has a monomial basis of $n$ partial compositions:
\begin{equation}
\label{associationtypes}
\begin{array}{l}
\mu \circ_1 \mu = \mu\big( \mu(x_1,\dots,x_n), x_{n+1}, \dots, x_{2n-1} \big), 
\quad \dots
\\[1pt]
\mu \circ_i \mu = \mu\big( x_1, \dots, x_{i-1}, \mu(x_i,\dots,x_{i+n-1}), \dots, x_{2n-1} \big), 
\quad \dots 
\\[1pt]
\mu \circ_n \mu = \mu\big( x_1, \dots, x_{n-1}, \mu(x_n,\dots,x_{2n-1}) \big).
\end{array}
\end{equation}

\begin{definition}
The \textit{weight} of a monomial is the number of occurrences of $\mu$.
Thus a monomial of weight $w$ has arity $N = 1 + w(n{-}1)$.
\end{definition}

\begin{lemma}
\label{Catalanlemma}
The dimension of $\Gamma(E)(N)$ for $N = 1 + w(n{-}1)$ is 
the number of plane rooted complete $n$-ary trees with $w$ internal nodes
(counting the root), and this is the $n$-ary Catalan number:
\[
\dim \Gamma(E)(N) = \frac{1}{1{+}(n{-}1)k} \binom{nk}{k}
\]
\end{lemma}

\begin{definition}
A \emph{quadratic relation} is an element of $\Gamma(E)(2n{-}1)$.
Let $(R) \subseteq \Gamma(E)$ be the operad ideal
generated by a subspace of quadratic relations $R \subseteq \Gamma(E)(2n{-}1)$.
The quotient operad $\Gamma(E)/(R)$ is a \emph{quadratic operad}.
\end{definition}

\begin{definition}
\label{tassociativeoperation}
\cite{Gnedbaye}.
The \textit{totally associative operad} with generator $\mu$ of arity $n$ and degree $d$
is the quadratic operad $\tass_d^n = \Gamma(E)/(R)$ where $(R)$ is spanned by
$\mu \circ_i \mu - \mu \circ_{i+1} \mu$ for $1 \le i \le n{-}1$.

\end{definition}

\begin{definition}
\label{passociativeoperation}
\cite{Gnedbaye}.
The \textit{partially associative operad} with generator $\mu$ of arity $n$ and degree $d$
is the quadratic operad $\pass_d^n = \Gamma(E)/(R)$ where $(R)$ is spanned by
\[
\sum_{i=1}^n (-1)^{(i+1)(n-1)} \mu \circ_i \mu.
\]
\end{definition}


\subsection{Koszul duality}

The following discussion is based on \cite{MR-1,MR-2}.
We write $(\mathbb{F})$ for the $\mathbb{Z}$-graded vector space with 
$(\mathbb{F})_0 = \mathbb{F}$ and $(\mathbb{F})_d = \{0\}$ for $d \ne 0$.
For any $\mathbb{Z}$-graded vector space $V = \bigoplus_{d \in \mathbb{Z}} V_d$,
the graded dual $V^\#$ is defined by
\[
V^\# = \mathrm{Hom}( V, (\mathbb{F}) ) = \bigoplus_{d \in \mathbb{Z}} (V^\#)_d,
\quad\quad\quad
(V^\#)_d = \mathrm{Hom}_d( V, (\mathbb{F}) ).
\]
Since $(\mathbb{F})$ is concentrated in degree 0, the only nonzero maps of degree $d$ 
from $V$ to $(\mathbb{F})$ are concentrated in degree $-d$:
\[
(V^\#)_d = \mathrm{Lin}( V_{-d}, \mathbb{F} ) = ( V_{-d} )^\ast,
\]
the ordinary vector space dual of $V_{-d}$.
For $E = \bigoplus_{a \ge 2} E(a)$ we obtain
\[
( E(a)^\# )_d =
\begin{cases}
\;
\mathbb{F}^\ast &\text{if $a = n$, $d = 0$} \\
\;
0 &\text{otherwise}
\end{cases}
\]
We define the $\mathbb{Z}$-graded vector space $E^\vee(a)$ by
\[
E^\vee(a) = {\uparrow^{a-2}} ( E(a)^\# ),
\]
where $\uparrow^{a-2}$ denotes the $(a{-}2)$-fold suspension.
Since $E(a)^\# = 0$ unless $a = n$,
\[
E^\vee(a) =
\begin{cases}
\;
 {\uparrow^{n-2}} ( E(n)^\# ) &\text{if $a = n$}
\\
\;
0 &\text{otherwise}
\end{cases}
\qquad
( E^\vee(a) )_d
=
\begin{cases}
\;
 \mathbb{F}^\ast &\text{if $a = n$, $d = n{-}2$}
\\
\;
0 &\text{otherwise}
\end{cases}
\]
Thus $\Gamma( E^\vee )$ is the free operad generated by the
dual operation $\mu^\ast$ in degree $n{-}2$:
if $n$ is odd (resp.~even) then $\Gamma( E^\vee )$ is generated by an odd
(resp.~even) operation.

To determine the $\mathbb{Z}$-graded vector space 
$R^\perp \subseteq \Gamma(E^\vee)(2n{-}1)$ of relations satisfied by $\mu^\ast$,
we consider the following morphism of graded vector spaces:
\begin{align*}
&
\langle -, - \rangle \colon
\Gamma(E^\vee)(2n{-}1) \otimes_{\mathbb{F}} \Gamma(E)(2n{-}1) \longrightarrow \mathbb{F},
\\
&
\langle \; {\uparrow^{n-2}} f_1^\ast \circ_i {\uparrow^{n-2}} g_1^\ast, \; f_2 \circ_j g_2 \; \rangle
=
\delta_{ij}(-1)^{(i+1)(n+1)} f_1^\ast(f_2) g_1^\ast(g_2)
\in
\mathbb{F}.
\end{align*}
It follows that the monomial basis of $\Gamma(E^\vee)(2n{-}1)$ is identical 
to that of \eqref{associationtypes} except that $\mu$ is replaced by $\mu^\ast$.
Hence the morphism
$\langle -, - \rangle$ has the simple form
\[
\langle \;
 ( \mu^\ast \circ_i \mu^\ast ), \;
( \mu \circ_j \mu )\;
\rangle
=
(-1)^{(i+1)(n+1)} \delta_{ij}.
\]
This is a $\mathbb{Z}$-graded pairing between $\Gamma(E^\vee)(2n{-}1)$
and $\Gamma(E)(2n{-}1)$.

\begin{definition}\label{Koszuldual}
The space of dual relations $R^\perp \subseteq \Gamma(E^\vee)(2n{-}1)$ 
is the orthogonal complement of $R \subseteq \Gamma(E)(2n{-}1)$
in the sense that
\[
R^\perp
=
\{ \,
\alpha^\ast \in \Gamma(E^\vee)(2n{-}1)
\mid
\alpha^\ast(\beta) = 0, \forall \, \beta \in R
\, \}.
\]
The \emph{Koszul dual} $\mathcal{P}^!$ of the original operad $\mathcal{P}$ is then 
$\mathcal{P}^! = \Gamma(E^\vee) / ( R^\perp )$.
\end{definition}

\begin{proposition}\label{Koszul of TA}
\cite[Proposition 14]{MR-2}
The Koszul dual of the totally associative operad $\tass^n_d$ 
generated by $n$-ary operation in degree $d$
is the partially associative operad $\pass^n_{-d+n-2}$ 
generated by an $n$-ary operation in degree $-d{+}n{-}2$.
\end{proposition}


\subsection{Operadic cohomology of $\mathcal{P}$-algebras}

We recall some definitions.

\begin{definition}
\cite[\S5.9.8]{LV}, \cite[\S1.1.20,\S3.71]{MSS} 
\label{algebra over the ns operad}
Let $\mathcal{P}$ be an ns operad.
Let $A$ be a $\mathbb{Z}$-graded vector space and let 
$\mathrm{End}(A) = \bigoplus_{n \ge 1} \mathrm{Hom}( A^{\otimes n}, A )$ 
be its (ns) endomorphism operad.
An operad morphism $\mathcal{P} \rightarrow \mathrm{End}(A)$ 
makes $A$ into a $\mathcal{P}$-\emph{algebra}.
Similarly, let $\mathrm{CoEnd}(A) = \bigoplus_{n \ge 1} \mathrm{Hom}( A, A^{\otimes n} )$ 
be its (ns) endomorphism cooperad.
An operad morphism $\mathcal{P} \rightarrow \mathrm{CoEnd}(A)$ 
makes $A$ into a $\mathcal{P}$-\emph{coalgebra}.
\end{definition}

\begin{definition}\cite[\S12.4.1]{LV}\label{cochaincomplex}
Let $\mathcal{P} = \mathcal{P}(E,R)$ be a quadratic operad and assume $A$ is a $\mathcal{P}$-algebra, and $M$ is an $A$-module.
We also assume that $A$ and $M$ are concentrated in degree 0.
An \emph{operadic cochain complex} is defined by:
\[
C_\mathcal{P}^\bullet(A,M):= Hom (\mathcal{P}^{\text{!}}(A),M)
\]
\end{definition}

%
%

\begin{remark}\cite[\S12.4.1]{LV}\label{operadic cohomoloy}
If we consider the $\mathcal{P}$-algebra $A$ as an $A$-module 
then we obtain the operadic cohomology $H_{\mathcal{P}}^\bullet(A,A)$ of $A$ 
with coefficients in $A$.
\end{remark}

\begin{theorem}\cite[Theorem 23]{Markl-2}\label{cup productaction}
Let $s(\mathcal{P} \otimes \mathcal{P}^{\text{!}})$ be the suspension of 
the operad $\mathcal{P} \otimes \mathcal{P}^{\text{!}}$.
There is a canonical action, called the cup product, 
of $s(\mathcal{P} \otimes \mathcal{P}^{\text{!}})$ 
on the graded vector space $C_\mathcal{P}^\bullet(A;A)$.
\end{theorem}


\section{Ternary case}

In this section we consider the operad $\tass_0^3$ generated by one totally associative
ternary operation $\mu$ of (homological) degree 0, 
and its Koszul dual $\pass_1^3$ generated by one partially associative ternary operation 
$\mu^\ast$ of (homological) degree 1.
That is, we show that the operadic complex
$C_\mathcal{P}^\bullet(A;A)$ admits a ternary cup product of degree 1.
Our goal is to understand explicitly the partially associative structure of the cohomology
of a totally associative algebra.
For a detailed study of this Koszul pair, see \cite{DMR,MR-1,MR-2}.
For results on algebras over the operad $\pass_0^3$ (that is, partially associative triple systems
with an operation of degree 0) see \cite{Bremner,GR}.


\subsection{The case $n = 3$}

Specializing the general discussion to the case $n = 3$, we obtain the following results.
First, $\Gamma(E)(2n-1)$, which is $\Gamma(E)(5)$ in ternary case, 
is isomorphic to the direct sum of three copies of the graded vector spaces 
corresponding to the three association types of arity five:
\begin{equation}\label{basis3aryoperation}
\begin{array}{l}
\mu \circ_1 \mu = \mu(\mu(x_1,x_2,x_3),x_4,x_5), \\[1pt]
\mu \circ_2 \mu = \mu(x_1,\mu(x_2,x_3,x_4),x_5), \\[1pt]
\mu \circ_3 \mu = \mu(x_1,x_2,\omega(x_3,x_4,x_5)),
\end{array}
\end{equation}
where $\circ_i$ denotes the usual operadic partial composition.
Thus $\dim \Gamma(E)(5) = 3$ with the monomial basis
$\mu \circ_1 \mu$, $\mu \circ_2 \mu$, $\mu \circ_3 \mu$.
Write $R \subseteq \Gamma(E)(5)$ for the graded subspace of quadratic relations
$\mu \circ_1 \mu - \mu \circ_2 \mu$,
$\mu \circ_1 \mu - \mu \circ_3 \mu$
corresponding to associative triple systems, 
and $(R) \subset \Gamma(E)$ for the operad ideal generated by $R$.
Write $\mathcal{P} = \Gamma(E) / (R)$ for the operad governing totally associative triple systems.
Let $V = E(a)$, We obtain
\[
( E(a)^\# )_d =
\begin{cases}
\Omega^\ast &\text{if $a = 3$, $d = 0$} \\
\{0\} &\text{otherwise}
\end{cases}
\]
where $\Omega \cong \mathbb{F}$.
Therefore the degree-graded vector space $E^\vee(a)$ would be:
\[
E^\vee(a) = {\uparrow^{a-2}} ( E(a)^\# ),
\]
where $\uparrow^{a-2}$ denotes the $(a{-}2)$-fold suspension of the degree-graded vector space $E(a)^\#$.
Since $E(a)^\# = \{0\}$ unless $a = 3$, we obtain
\[
E^\vee(a) =
\begin{cases}
 {\uparrow} ( E(3)^\# ) &\text{if $a = 3$}
\\
\{0\} &\text{otherwise}
\end{cases}
\]
By definition of suspension, this gives
\[
( E^\vee(a) )_d
=
\begin{cases}
 \Omega^\ast &\text{if $a = 3$, $d = 1$}
\\
\{0\} &\text{otherwise}.
\end{cases}
\]
Thus $E^\vee(3)$ is 1-dimensional, concentrated in degree 1, and
$E^\vee(a)$ is 0-dimensional for $a \ne 3$.
Therefore $\Gamma( E^\vee )$ is the free operad generated by the
dual operation $\mu^\ast$ placed in degree 1;
that is, $\Gamma( E^\vee )$ is generated by an odd operation.

We next determine the degree graded subspace $R^\perp \subseteq \Gamma(E^\vee)(5)$ 
of relations satisfied by the generating operation $\mu^\ast$.
(Note that these relations are quadratic and hence have homological degree 2;
that is, they represent even operations.)
Consider the following morphism of graded vector spaces:
\[
\langle -, - \rangle \colon
\Gamma(E^\vee)(5) \otimes_{\mathbb{F}} \Gamma(E)(5) \longrightarrow \mathbb{F},
\]
defined by the equation
\[
\langle \; {\uparrow} f_1^\ast \circ_i {\uparrow} g_1^\ast, \; f_2 \circ_j g_2 \; \rangle
=
\delta_{ij} f_1^\ast(f_2) g_1^\ast(g_2)
\in
\mathbb{F}.
\]
In other words, if we imitate the monomial basis \eqref{basis3aryoperation} for $\Gamma(E)(5)$, we obtain
the following monomial basis of $\Gamma(E^\vee)(5)$,
\begin{equation}
\label{dualquadraticmonomialbasis}
\begin{array}{l}
\mu^\ast \circ_1 \mu^\ast = \mu^\ast(\mu^\ast(x_1,x_2,x_3),x_4,x_5), \\[1pt]
\mu^\ast \circ_2 \mu^\ast = \mu^\ast(x_1,\mu^\ast(x_2,x_3,x_4),x_5), \\[1pt]
\mu^\ast \circ_3 \mu^\ast = \mu^\ast(x_1,x_2,\mu^\ast(x_3,x_4,x_5)),
\end{array}
\end{equation}
For this basis of $\Gamma(E^\vee)(5)$, the graded morphism
$\langle -, - \rangle$ takes the simple form
\[
\langle \;
 ( \mu^\ast \circ_i \mu^\ast ), \;
( \mu \circ_j \omega ) \;
\rangle
=
\delta_{ij}.
\]
We now define $R^\perp \subseteq \Gamma(E^\vee)(5)$ to be 
the orthogonal complement of $R \subseteq \Gamma(E)(5)$:
\[
R^\perp
=
\{ \,
\alpha^\ast \in \Gamma(E^\vee)(5)
\mid
\alpha^\ast(\beta) = 0, \forall \, \beta \in R
\, \}.
\]
The Koszul dual $\mathcal{P}^!$ of the original operad $\mathcal{P}$ is then defined by
\[
\mathcal{P}^! = \Gamma(E^\vee) / ( R^\perp ).
\]
Thus $\tass_0^3$ is generated by the ternary operation $\mu$ which satisfies
\[
\mu \circ_1 \mu =\mu \circ_2 \mu = \mu \circ_3 \mu.
\]
Similarly, $\pass_1^3$ is generated by the ternary operation $\nu$ which satisfies
\[
\mu \circ_1 \mu + \mu \circ_2 \mu + \mu \circ_3 \mu = 0.
\]

\begin{lemma} \label{KoszuldualofTA3}
Let $\mathcal{P}$ be a totally associative operad $\tass^3_0$ 
generated by the ternary operation $\mu$ in degree 0.
Then its Koszul dual $\mathcal{P}^{\text{!}}$ is the partially associative operad
$\pass^3_1$ generated by the ternary operation $\mu^\ast$ of degree 1 which satisfies
 \[
 \mu^\ast(\mu^\ast(-,-,-),-,-) + \mu^\ast(-,\mu^\ast(-,-,-),-) + \mu^\ast(-,-,\mu^\ast(-,-,-))=0
 \]
\end{lemma}

\begin{proof}
The operad $\tass^3_0$ is the quadratic operad $\Gamma(E)/ (R)$ where
$R \subseteq \Gamma(E)(5)$ is the subspace of quadratic relations generated by
$\mu \circ_1 \mu - \mu \circ_2 \mu$ and
$\mu \circ_1 \mu - \mu \circ_3 \mu$.
So to find
$R^\perp
=
\{ \,
\alpha^\ast \in \Gamma(E^\vee)(5)
\mid
\alpha^\ast(\beta) = 0, \forall \, \beta \in R
\, \}$,
it suffices to find the nullspace of the matrix whose rows
are the coefficient vectors of the two relations:
\[
\left[ \begin{array}{rrr}
 1 & -1 &  0\\
 1 &  0 & -1
\end{array} \right]
\]
A basis for the nullspace is $[1,1,1]$ and hence
$\alpha^\ast = \mu^\ast \circ_1 \mu^\ast + \mu^\ast \circ_2 \mu^\ast + \mu^\ast \circ_3 \mu^\ast $.
So the ternary operation $\mu^\ast$ of $\Gamma(E^\vee) / ( R^\perp )$ satisfies the claimed relation.
\end{proof}

\begin{notation}
For simplicity we denote the three generators
$\mu^\ast \circ_1 \mu^\ast$, $\mu^\ast \circ_2 \mu^\ast$, and $\mu^\ast \circ_3 \mu^\ast$ of
$\mathcal{P}^{\text{!}(5)}$ by $m_1,\,m_2,\, m_3$.
\end{notation}


\subsection{Ternary cup product on the cochain complex}

Let $\mathcal{P}$ be the ns associative operad generated by a ternary operation in degree 0 
and let $A$ be a $\mathcal{P}$-algebra.
For any even number $n$, $\mathcal{P}(n)=0$.
If we consider $A$ as an $A$-module, then by definition \ref{cochaincomplex}, 
$C_\mathcal{P}^n(A,A):= Hom (\mathcal{P}^{\text{!}^{(2n+1)}}(A),A)$.
So the first three classes of cochain complexes would be equal to:
\begin{align*}
C_\mathcal{P}^0(A,A) &= Hom (\mathcal{P}^{\text{!}^{(1)}}(A),A))\cong Hom(\mathcal{P}^{\text{!}}(1),Hom (A,A) )
\\
C_\mathcal{P}^1(A,A) &= \mathrm{Hom} (\mathcal{P}^{\text{!}^{(3)}}(A),A))\cong \mathrm{Hom}(\mathcal{P}^{\text{!}}(3),\mathrm{Hom} (A^{\otimes 3},A) )
\\
C_\mathcal{P}^2(A,A) &= \mathrm{Hom} (\mathcal{P}^{\text{!}^{(5)}}(A),A))\cong \mathrm{Hom}(\mathcal{P}^{\text{!}}(5),\mathrm{Hom} (A^{\otimes 5},A))
\end{align*}
The second isomorphisms comes from the following natural isomorphism:
\[
\mathrm{Hom}(\mathcal{P}(n),\mathrm{Hom} (A^{\otimes n},A) ) \cong \mathrm{Hom}(\mathcal{P}(n)\otimes A^{\otimes n},A).
\]

\begin{remark}\label{elementofchaincopmlex}
 By the Lemma \ref{KoszuldualofTA3}, the operation $\mu^{\ast}$ is ternary and $\mathcal{P}^{\text{!}}(3)= \langle \mu^{\ast} \rangle$ and $\mathcal{P}^{\text{!}}(5)$ is generated by generators $m_1,\,m_2,\, m_3$ such that $m_1, + \,m_2, +\, m_3 =0$.
So two operations of these three, like $m_1,\,m_2$ form the basis of $\mathcal{P}^{\text{!}}(5)$.
If $f$ in $C_\mathcal{P}^0(A,A)$, it means $f \in Lin(A,A)$.
Also if $f$ in $C_\mathcal{P}^1(A,A)$, then $f(\mu ^{\ast})$ is an element of $Lin(A^{\otimes 3},A)$.
So there are three linear homeomorphisms $f_1, f_2, f_3: A \rightarrow A$ corresponding to $f(\mu ^{\ast}):A^{\otimes 3} \rightarrow A$ that:
\[
f(\mu ^{\ast})(a_1,a_2,a_3)= \mu (f_1(a_1), f_2(a_2), f_3(a_3)).
\]
Let $f$ be an element of $C_\mathcal{P}^2(A,A)= \mathrm{Hom}(\mathcal{P}^{\text{!}}(5),\mathrm{Hom} (A^{\otimes 5},A))$, then $f(m_j)$ for $j=1,2$ are elements of $Lin(A^{\otimes 5},A)$.
 So there should be five linear homeomorphism $f_i: A \rightarrow A,\; i=1,\dots 5$ corresponding to $f$ such that
\begin{align*}
f(m_1)(a_1,a_2,a_3,a_4,a_5) &= \mu \big( \mu ( f_1(a_1), f_2(a_2), f_3(a_3)), f_4(a_4), f_5(a_5)\big),
\\
f(m_2)(a_1,a_2,a_3,a_4,a_5) &= \mu \big( f_1(a_1), \mu (f_2(a_2), f_3(a_3), f_4(a_4)), f_5(a_5)\big).
\end{align*}
\end{remark}

\begin{definition}\label{cupproductincochaincomplex1}
Let $C_\mathcal{P}^\bullet(A,A):= \mathrm{Hom} (\mathcal{P}^{\text{!}}(A),A)$ be the operadic cochain complex for totally associative operad
$\mathcal{P}$ generated by $\mu$.
Let $f,g,h$ be three elements of $C_\mathcal{P}^\bullet(A,A)$ with degree $n_1$, $n_2$, $n_3$.
Define the \emph{ternary cup product} by
\[
\langle -,-,-\rangle: C_\mathcal{P}^\bullet(A,A) \otimes C_\mathcal{P}^\bullet(A,A) \otimes C_\mathcal{P}^\bullet(A,A) \rightarrow C_\mathcal{P}^\bullet(A,A),
\]
by the equation
\[
\langle f,g,h\rangle (\gamma)(a_1, \dots, a_{2(k) +1})= \gamma(f \otimes g\otimes h)(a_1, \dots ,a_{2k+1}),
\]
where $\langle f,g,h\rangle \in C_\mathcal{P}^{k}(A,A)$ for $k=n_1 +n_2 +n_3 +1$
and $\gamma \in \mathcal{P}^{\text{!}}(2(k)+1)$.
\end{definition}

For a fixed generator $\gamma$ in $\mathcal{P}^{\text{!}}(2k+1)$, as $f,g,h$ are elements of $C_\mathcal{P}^\bullet(A,A)$ of degree $n_1$, $n_2$ and $n_3$, they act on the elements of
$\mathcal{P}^{\text{!}}(2n_1+1)$, $\mathcal{P}^{\text{!}}(2 n_2+1)$, and $\mathcal{P}^{\text{!}}(2 n_3+1)$ respectively.
So by Remark \ref{elementofchaincopmlex}, there are linear homeomorphism $f_i,\, i=1,\dots, 2n_1+1$, $g_j,\, j=1,\dots, 2n_2+1$, and $h_l,\, l=1, \dots, 2n_3+1$ elements of $Lin( A,A)$ corresponding to $f,g,h$, respectively.
Therefore the definition makes sense.

\begin{example}
Assume $f,g,h$ are elements of $C_\mathcal{P}^0(A,A)$. Then
$\langle f,g,h\rangle\in C_\mathcal{P}^1(A,A)$ and $\langle f,g,h\rangle (\mu^{\ast})$ would be an element of $Lin(A^{\otimes 3},A)$, so by last definition:
\[
\langle f,g,h\rangle (\mu^{\star})(a_1,a_2,a_3)= \mu (f(a_1), g(a_2), h(a_3))
\]
\end{example}

\begin{example}
Let us to describe the ternary product $\langle -,-,-\rangle$, when $f$ has degree 1 and both of $g,h$ have degree 0.
In this case $\langle f,g,h\rangle$ would be element of
$C_\mathcal{P}^2(A,A)$, so
$\langle f,g,h\rangle(m_i)$ for $i=1,2$ is in $Lin(A^{\otimes 5},A)$.
But $f$ has degree 1, by Remark \ref{elementofchaincopmlex}, there exist three linear homeomorphism $f_1, f_2, f_3: A \rightarrow A$ corresponding to $f$.
Therefore
\begin{align*}
\langle f,g,h\rangle (m_1)(a_1,a_2,a_3,a_4,a_5)& = \mu \big( \mu (f(a_1), f(a_2), f(a_3)), g(a_4), h(a_5)),\\
\langle f,g,h\rangle (m_2)(a_1,a_2,a_3,a_4,a_5)& = \mu \big((f(a_1), \mu (f(a_2), f(a_3), g(a_4)), h(a_5)).
\end{align*}
\end{example}


\subsection{$\pass_1^3$-algebra structure of cochain complex of}

\begin{theorem}\label{cochainispartially1}
Let $\mathcal{P}$ be a ns totally associative operad generated by ternary generator $\mu$ 
in degree 0, let $A$ be a $\mathcal{P}$-algebra, and $C_\mathcal{P}^\bullet(A,A)$ 
the operadic cochain complex (Definition \ref{cochaincomplex}).
The ternary cup product $\langle -,-,- \rangle$ on $C_\mathcal{P}^\bullet(A,A)$ is a partially associative product of degree 1.
\end{theorem}

\begin{proof}
We need to show that
\[
\langle \langle f,g,h \rangle, k,l \rangle + \langle f,\langle g,h,k \rangle,l\rangle + \langle f,g, \langle h,k,l\rangle\rangle =0,
\]
where $f,g,h,k,l$ are elements of $C_\mathcal{P}^\bullet(A,A)$ of degree $n_1,n_2,n_3,n_4,n_5$, respectively.

We prove by induction.
So first, we check our claim for the simplest case, i.e $f,g,h,k,l$ are of degree zero, so they are elements of $Lin(A,A)$.
By our definition $\langle f,g,h \rangle$, $\langle g,h,k \rangle$, $ \langle h,k,l \rangle$ are of degree one, so they belong to
$C_\mathcal{P}^1(A,A)$.
Also $\langle \langle f,g,h \rangle, k,l \rangle$ and the two other terms of above equality are of degree two, so they are elements of $C_\mathcal{P}^3(A,A)$.
By remark \ref{elementofchaincopmlex}, $\langle \langle f,g,h \rangle, k,l \rangle (m_j):A^{\otimes 5} \rightarrow A$.
Therefore it is enough to show that
\[
\big(\langle \langle f,g,h \rangle, k,l \rangle + \langle f,\langle g,h,k \rangle,l\rangle + \langle f,g, \langle h,k,l\rangle\rangle\big) (m_j)=0,\; j=1,2.
\]
Let us try first for $m_1$:
\begin{align*}
\big(\langle\langle f,g,h\rangle,k,l\rangle (m_1)\big)(a_1,a_2,a_3,a_4,a_5)&:= \mu (\mu (f(a_1), g(a_2), h(a_3)), k(a_4), l(a_5))\\
\big(\langle f,\langle g,h,k\rangle,l\rangle (m_1)\big)(a_1,a_2,a_3,a_4,a_5)&:= \mu ( f(a_1),\mu ( g(a_2), h(a_3), k(a_4)), l(a_5))\\
\big(\langle f,g,\langle h,k,l\rangle\rangle (m_1)\big)(a_1,a_2,a_3,a_4,a_5)&:= \mu ( f(a_1), g(a_2), \mu (h(a_3), k(a_4), l(a_5)))
\end{align*}
But the second and third equality are also equal respectively to 
\[
\langle\langle f,g,h\rangle,k,l\rangle (m_2)(a_1,a_2,a_3,a_4,a_5),
\qquad
\langle\langle f,g,h\rangle,k,l\rangle (m_3)(a_1,a_2,a_3,a_4,a_5).
\] 
On the other hand, by Lemma \ref{KoszuldualofTA3}, we assumed $m_1+m_2+m_3=0$, so:
\begin{align*}
&\big(\langle\langle f,g,h\rangle,k,l\rangle + \langle f,\langle g,h,k\rangle,l\rangle + \langle f,g,\langle h,k,l\rangle\rangle )\big)(m_1)(a_1,a_2,a_3,a_4,a_5)\\
&= m_1 (f(a_1), g(a_2), h(a_3), k(a_4), l(a_5))+
m_2 (f(a_1), g(a_2), h(a_3), k(a_4), l(a_5))\\
&+ m_3 (f(a_1), g(a_2), h(a_3), k(a_4), l(a_5))\\
&= (m_1+m_2+m_3)(f(a_1), g(a_2), h(a_3), k(a_4), l(a_5))=0
\end{align*}
By induction, assume if $f,g,h,k,l$ are elements of degree $n_1,\dots,n_5$, respectively,and $ n_1+\dots+n_5=k$. Then
\[
\big(\langle \langle f,g,h \rangle, k,l \rangle + \langle f,\langle g,h,k \rangle,l\rangle + \langle f,g, \langle h,k,l\rangle\rangle\big) (m_j)(a_1,\dots,a_{2k+1})=0,
\]
for all $m_j$ of the basis of $\mathcal{P}^{\text{!}}(2k+1)$.
Now we show, if $f,g,h,k,l$ are elements of degree $n'_1,\dots,n'_5$ respectively, and $ n'_1+\dots+n'_5 =k +1$. Then
\[
\big(\langle \langle f,g,h \rangle, k,l \rangle + \langle f,\langle g,h,k \rangle,l\rangle + \langle f,g, \langle h,k,l\rangle\rangle\big) (m'_j) (a_1,\dots,a_{2k+3})=0,
\]
for all $m'_j$ of the basis of $\mathcal{P}^{\text{!}}(2k+3)$.
Any $m'_j$ is equal to $m_j \circ_i \mu^\ast$, such that $m_j$ is an element of the basis for $\mathcal{P}^{\text{!}}(2k+1)$.
So
\begin{align*}
&
\big(\langle \langle f,g,h \rangle, k,l \rangle + 
\langle f,\langle g,h,k \rangle,l\rangle + 
\langle f,g, \langle h,k,l\rangle\rangle\big) (m'_j)
\\
&= 
\big(\langle \langle f,g,h \rangle, k,l \rangle + 
\langle f,\langle g,h,k \rangle,l\rangle + 
\langle f,g, \langle h,k,l\rangle\rangle\big) (m_j \circ_i \mu^\ast)(a_1,\dots,a_{2k+3})
\\
&=
\big(\langle \langle f,g,h \rangle, k,l \rangle + 
\langle f,\langle g,h,k \rangle,l\rangle 
\\
&\qquad
+ 
\langle f,g, \langle h,k,l\rangle\rangle\big) (m_j)(a_1,\dots, \mu^\ast(a_i,a_{i+1}, a_{i+2}), \dots a_{2k+3}).
\end{align*}
But by induction assumption:
\[
\big(\langle \langle f,g,h \rangle, k,l \rangle + \langle f,\langle g,h,k \rangle,l\rangle + \langle f,g, \langle h,k,l\rangle\rangle\big) (m_j)(a_1,\dots,a_{2k+1}) =0
\]
for any $(a_1,\dots,a_{2k+1})$ in $A^{\otimes 2k+1}$.
So
\begin{align*}
&
\big(\langle \langle f,g,h \rangle, k,l \rangle + 
\langle f,\langle g,h,k \rangle,l\rangle + 
\\
&\qquad
\langle f,g, \langle h,k,l\rangle\rangle\big) (m_j)(a_1,\dots, \mu^\ast(a_i,a_{i+1}, a_{i+2}), \dots a_{2k+3})
= 0.
\end{align*}
The definition of ternary cup product $\langle -,-,- \rangle$ show that, this ternary operation of $C_\mathcal{P}^\bullet(A,A)$ has degree one.
\end{proof}


\section{Quaternary case}

The operad $\tass_0^4$ is generated by quaternary operation
$\mu$ that satisfies:
\[
\mu (\mu, id,id,id)=\mu (id, \mu,id,id)= \mu (id, id,\mu,id)= \mu (id, id,id,\mu).
\]
If $d=2$ and $n=4$, the operad $\pass_2^4$ is generated by quaternary operation $\mu$,
that by Example \ref{passociativeoperation} satisfies:
\[
\mu (\mu, id,id,id) - \mu (id, \mu,id,id) + \mu (id, id,\mu,id)- \mu(id, id,id,\mu)=0.
\]
These two operads were originally studied by Gnedbaye \cite[\S 3.4]{Gnedbaye}.
It follows from his results that if $n$ is even then the free ns operad
generated by one $(n-1)$-ary operation and the partially associative $n$-ary operad $\pass^n_0$
are linearly isomorphic (that is, as weight-graded vector spaces).

\begin{remark}
\label{padimension}
The reason for this phenomenon is that for even $n$, the single relation defining partial
associativity forms a Gr\"obner basis of the operadic ideal it generates,
for the path glex order.
This can be verified by computing a Gr\"obner basis for the Koszul dual operad,
which is the (de)suspension of the totally associative $n$-ary operad.
It follows that the monomials in normal form are precisely those where partial composition is
forbidden in the first position, so only the last $n-1$ partial compositions are permitted.
This gives the required weight-preserving linear isomorphism.
In the odd case, as we have already seen for $n = 3$, the Gr\"obner basis is more complicated;
in fact, the corresponding operad is not Koszul \cite{DMR}.
It follows that the dimensions of the homogeneous spaces in the partially associative quaternary
operad are given by the ternary Catalan numbers (Proposition \ref{Catalanlemma}).
\end{remark}

\begin{lemma} \label{KoszuldualofTA4}
Let $\mathcal{P}$ be a totally associative operad $\tass^4_0$ generated by
the quaternary operation $\mu$ in degree zero.
Then its Koszul cooperad $\mathcal{P}^{\text{!}}$ is the partially associative operad
$\pass^4_2$ generated by quaternary operation $\mu^\ast=(-,-,-,-)$ of degree two,
which is partially associative, so it satisfies the relation:
 \[
\mu^\ast (\mu^\ast, id,id,id) - \mu^\ast (id, \mu^\ast,id,id) + \mu^\ast (id, id,\mu^\ast,id)- \mu^\ast(id, id,id,\mu^\ast)=0.
\]
\end{lemma}


\subsection{Quaternary cup product on the cochain complex}

Let $\mathcal{P}$ be the ns operad generated by a totally associative quaternary operation in degree zero
and let $A$ be a $\mathcal{P}$-algebra.
We note that $\mathcal{P}(n) \ne 0$ if and only if $n \equiv 1$ (mod 3).
If we consider $A$ as a module over itself, then by Definition \ref{cochaincomplex},
\[
C_\mathcal{P}^n(A,A) := Hom (\mathcal{P}^{\text{!}^{(3n+1)}}(A),A).
\]
Recall the the following natural isomorphism for operads:
\[
Hom(\mathcal{P}(n),Hom (A^{\otimes n},A) ) \cong Hom(\mathcal{P}(n)\otimes A^{\otimes n},A).
\]

\begin{example}
The first five homogeneous spaces in the cochain complex are
\begin{align*}
C_\mathcal{P}^0(A,A)&= Hom (\mathcal{P}^{\text{!}^{(1)}}(A),A))\cong Hom(\mathcal{P}^{\text{!}}(1),Hom (A,A) ),\\
C_\mathcal{P}^1(A,A)&= Hom (\mathcal{P}^{\text{!}^{(4)}}(A),A))\cong Hom(\mathcal{P}^{\text{!}}(4),Hom (A^{\otimes 4},A) ),\\
C_\mathcal{P}^2(A,A)&= Hom (\mathcal{P}^{\text{!}^{(7)}}(A),A))\cong Hom(\mathcal{P}^{\text{!}}(7),Hom (A^{\otimes 7},A)),\\
C_\mathcal{P}^3(A,A)&= Hom (\mathcal{P}^{\text{!}^{(10)}}(A),A))\cong Hom(\mathcal{P}^{\text{!}}(10),Hom (A^{\otimes 10},A)),\\
C_\mathcal{P}^4(A,A)&= Hom (\mathcal{P}^{\text{!}^{(13)}}(A),A))\cong Hom(\mathcal{P}^{\text{!}}(13),Hom (A^{\otimes 13},A)).
\end{align*}
\end{example}

By Lemma \ref{KoszuldualofTA4}, the dual operation $\mu^{\star}$ is quaternary where
$\mathcal{P}^{\text{!}}(4)= \langle \mu^{\star} \rangle$, and
$\mathcal{P}^{\text{!}}(7)$ is spanned by the monomials $m_1,\,m_2,\, m_3, m_4$ satisfying the partially associativity relation
$m_1 - \,m_2 +\, m_3 - \, m_4 =0$.
So we may choose three of these four operations, say $m_1,\,m_2,\,m_3$ as the basis of $\mathcal{P}^{\text{!}}(7)$.

\begin{example}
\label{chaincomplex4}
If $f \in C_\mathcal{P}^0(A,A)$ then $f \in Lin(A,A)$ since $\dim \mathcal{P}^{\text{!}}(1) = 1$.

If $f \in C_\mathcal{P}^1(A,A)$, then $f(\mu ^{\star}) \in Lin(A^{\otimes 4},A)$ since $\dim \mathcal{P}^{\text{!}}(4) = 1$.
So there are four $A$-module morphisms $f_1, f_2, f_3, f_4 \colon A \rightarrow A$ corresponding to the four components
of $f(\mu ^{\star}) \colon A^{\otimes 4} \rightarrow A$ regarded as a multilinear map:
\[
f(\mu ^{\star})(a_1,a_2,a_3,a_4)= \mu (f_1(a_1), f_2(a_2), f_3(a_3), f_4(a_4)).
\]

If $f \in C_\mathcal{P}^2(A,A)= Hom(\mathcal{P}^{\text{!}}(7),Hom (A^{\otimes 7},A))$,
then since $\dim \mathcal{P}^{\text{!}}(7) = 3$,
we have $f(m_j) \in Lin(A^{\otimes 7},A)$ for $j = 1,2,3$.
So there are seven $A$-module morphisms $f_i: A \rightarrow A,\; i=1,\dots 7$ such that:
\begin{align*}
f(m_1)(a_1,\dots,a_7) &= \mu\Big( \mu\big( f_1(a_1), f_2(a_2), f_3(a_3), f_4(a_4) \big), f_5(a_5), f_6(a_6), f_7(a_7) \Big),
\\
f(m_2)(a_1,\dots,a_7) &= \mu\Big( f_1(a_1), \mu\big( f_2(a_2), f_3(a_3), f_4(a_4), f_5(a_5) \big), f_6(a_6), f_7(a_7) \Big),
\\
f(m_3)(a_1,\dots,a_7) &= \mu\Big( f_1(a_1), f_2(a_2), \mu\big( f_3(a_3), f_4(a_4), f_5(a_5), f_6(a_6) \big), f_7(a_7) \Big).
\end{align*}
\end{example}

\begin{definition}\label{cupproductincochaincomplex2}
Let $C_\mathcal{P}^\bullet(A,A):= Hom (\mathcal{P}^{\text{!}}(A),A)$ be the operadic cochain complex for
the totally associative operad $\mathcal{P}$ generated by the quaternary operation $\mu$.
Let $f,g,h,k$ be four elements of $C_\mathcal{P}^\bullet(A,A)$ with degrees $n_1$, $n_2,n_3$, $n_4$ respectively.
The \emph{quaternary cup product} is a map $\langle -,-,-,-\rangle$ of (homological) degree 2:
\[
\langle -,-,-,-\rangle:
 C_\mathcal{P}^\bullet(A,A) \otimes C_\mathcal{P}^\bullet(A,A) \otimes C_\mathcal{P}^\bullet(A,A)\otimes C_\mathcal{P}^\bullet(A,A) \longrightarrow C_\mathcal{P}^\bullet(A,A).
\]
In order to define the element $\langle f,g,h,k\rangle$ of $C_\mathcal{P}^{r}(A,A)$ where $r=n_1 +n_2 +n_3 + n_4+2$,
we make the following observations:
\begin{itemize}[leftmargin=*]
\item
The elements $f, g, h, k$ are determined by their components, which are $A$-module morphisms $f_i, g_j,h_l,k_t \in Lin( A,A)$, where
\begin{align*}
&
1 \le i \le 3n_1+1,
\\
&
3n_1+2 \le j \le 3n_1+3n_2+2,
\\
&
3n_1+3n_2+3 \le l \le 3n_1+3n_2+3n_3+3,
\\
&
3n_1+3n_2+3n_3+4 \le t \le 3n_1+3n_2+3n_3+3n_4+4.
\end{align*}
Note that
\[
3n_1{+}3n_2{+}3n_3{+}3n_4{+}4 = (3n_1{+}1) + (3n_2{+}1) + (3n_3{+}1) + (3n_4{+}1) = 3r-2.
\]
Thus the factors in the tensor product are indexed as follows:
\[
f_1 \otimes \cdots \otimes
g_{3n_1+2} \otimes \cdots \otimes
h_{3n_1+3n_2+3} \otimes \cdots \otimes
k_{3n_1+3n_2+3n_3+4} \otimes \cdots \otimes k_{3r-2}.
\]
\item
Since any element $\gamma \in \mathcal{P}^{\text{!}}(3r+1)$ can be written in the form $\gamma = \alpha \circ_s \mu^\ast$ where
$\alpha \in \mathcal{P}^{\text{!}}(3r-2)$ and $1 \le s \le 3r-2$, we can define $\langle f,g,h,k\rangle (\gamma)$
inductively by weight.
For each partial composition index $s$, we need to describe the map corresponding to position $s$.
For example, if position $s$ corresponds to tensor factor $k_j$ then
\begin{align*}
&\langle f,g,h,k\rangle (\gamma)(a_1, \dots, a_{3r +1})\\
&= \alpha \circ_s \mu^\ast (f_1 \otimes \cdots f_{3n_1+1}\cdots g_{3n_2+1} \cdots h_{3n_3+1}\cdots k_j^{\otimes 4} \cdots \otimes k_{3r-2})(a_1 \cdots a_{3r +1})\\
&=\alpha \circ_s \mu^\ast\big(f_1(a_1),\dots, k_j(a_i), k_j(a_{i+1}) , k_j(a_{i+2}), k_j(a_{i+3}), \dots k_{3r-2}(a_{3r+1})\big)\\
&=\alpha\Big(f_1(a_1),\dots, \mu\big(k_j(a_i),k_j(a_{i+1}), k_j(a_{i+2}),k_j(a_{i+3})\big),\dots, k_{3r-2}(a_{3r+1})\Big).
\end{align*}
\end{itemize}
\end{definition}

\begin{example}
In the simplest case, let $f,g,h,k$ be four elements of $C_\mathcal{P}^\bullet(A,A)$ of degree 0.
It follows that
\[
\langle f,g,h,k\rangle \in C_\mathcal{P}^2(A,A),
\qquad
\langle f,g,h,k\rangle (\mu^{\star})\in Lin(A^{\otimes 7},A).
\]
Using the basis $\{m_1,m_2,m_3\}$ of $\mathcal{P}^{\text{!}}(7)$, for $s = 1, 2, 3$ we obtain:
\begin{align*}
\langle f,g,h,k\rangle (m_1)(a_1,\dots,a_7)&= \mu^\ast\circ_1 \mu^\ast \big(f^{\otimes 4}\otimes g\otimes h\otimes k\big)(a_1,\dots,a_7)\\
&=\mu \Big( \mu \big(f(a_1), f(a_2), f(a_3),f(a_4)\big), g(a_5), h(a_6), k(a_7)\Big).\\
\langle f,g,h,k\rangle (m_2)(a_1,\dots,a_7)&= \mu^\ast\circ_2 \mu^\ast \big(f\otimes g^{\otimes 4}\otimes h\otimes k\big)(a_1,\dots,a_7)\\
&=\mu \Big(f(a_1), \mu \big( g(a_2), g(a_3),g(a_4), g(a_5)\big), h(a_6), k(a_7)\Big).\\
\langle f,g,h,k\rangle (m_3)(a_1,\dots,a_7)&= \mu^\ast\circ_3 \mu^\ast \big(f\otimes g\otimes g^{\otimes 4}\otimes k\big)(a_1,\dots,a_7)\\
&=\mu \Big(f(a_1), g(a_2), \mu \big( h(a_3),h(a_4), h(a_5), h(a_6)\big), k(a_7)\Big).
\end{align*}
\end{example}


\subsection{$\pass_2^4$-algebra structure of cochain complex}

\begin{theorem}\label{cochainispartially2}
Let $\mathcal{P}$ be the ns totally associative operad generated by the quaternary operation $\mu$ of (homological) degree 0,
let $A$ be a $\mathcal{P}$-algebra, and let $C_\mathcal{P}^\bullet(A,A)$ be its operadic cochain complex.
Then the quaternary cup product of Definition \ref{cupproductincochaincomplex2} on
$C_\mathcal{P}^\bullet(A,A)$ is a quaternary partially associative product of degree 2.
\end{theorem}

\begin{proof}
We need to show that
\begin{equation}
\label{tobeproved}
\begin{array}{l}
\langle \langle f,g,h,k \rangle, l,u,v \rangle
- \langle f,\langle g,h,k,l \rangle,u,v\rangle
\\[1mm]
{}
+ \langle f,g, \langle h,k,l,u\rangle,v\rangle
- \langle f,g,h, \langle k,l,u,v\rangle = 0,
\end{array}
\end{equation}
for arbitrary elements $f,\dots,v$ in $C_\mathcal{P}^\bullet(A,A)$ of degrees $n_1,\dots,n_7$, respectively.
We proceed by induction on the sum $n_1 + \cdots + n_7$.

For the basis, we verify the simplest case in which $f,\dots,v$ have degree 0: that is, they are in $Lin(A,A)$.
The four interior products $\langle f,g,h,k \rangle$, $\langle g,h,k,l \rangle$, $\langle h,k,l,u \rangle$, $ \langle k,l,u,v \rangle$
have degree 2; that is, they are in $C_\mathcal{P}^2(A,A)$.
The four terms in relation \eqref{tobeproved} have degree 4; that is, they are in $C_\mathcal{P}^4(A,A)$.
By Example \ref{chaincomplex4}, for every basis monomial $\beta \in \mathcal{P}^{\text{!}}(13)$, we have
\[
\langle \langle f,g,h,k \rangle, l,u,v \rangle (\beta):A^{\otimes 13} \longrightarrow A.
\]
Therefore it is enough to verify the relation for every basis monomial $\beta$; that is,
\begin{align*}
&
\langle \langle f,g,h,k \rangle, l,u,v \rangle (\beta)
- \langle f,\langle g,h,k,l \rangle,u,v\rangle (\beta)
\\
&\quad
+ \langle f,g, \langle h,k,l,u\rangle,v\rangle (\beta)
- \langle f,g,h \langle k,l,u,v\rangle (\beta) = 0.
\end{align*}

We work out the details for the monomial $\beta = ( m_1 \circ_1 \mu^\ast ) \circ_9 \mu^\ast$
obtained by applying partial compositions to $m_1$:
\begin{align*}s
&
\big(\langle \langle f,g,h,k \rangle, l,u,v \rangle (\beta) \big)(a_1,\dots,a_{13})
=
\beta(\langle f,g,h,k \rangle \otimes l\otimes u \otimes v)(a_1,\dots,a_{13})
\\
&=
\big( (m_1\circ_1\mu^\ast)\circ_9 \mu^\ast \big) (\langle f,g,h,k \rangle \otimes l\otimes u \otimes v)(a_1,\dots,a_{13})
\\
&=
\big( (m_1\circ_1\mu^\ast)\circ_9 \mu^\ast \big)(f^{\otimes 4}\otimes g \otimes h \otimes k \otimes l\otimes u^{\otimes 4} \otimes v)(a_1,\dots,a_{13}),
\\
&
\big(\langle f,\langle g,h,k,l \rangle,u,v\rangle (\beta) \big)(a_1,\dots,a_{13})
=
\beta(f\otimes \langle g,h,k,l\rangle \otimes u \otimes v)(a_1,\dots,a_{13})
\\
&=
\big( (m_1\circ_1\mu^\ast)\circ_9 \mu^\ast \big) (f\otimes \langle g,h,k,l\rangle \otimes u \otimes v)(a_1,\dots,a_{13})
\\
&=
\big( (m_1\circ_1\mu^\ast)\circ_9 \mu^\ast \big)(f\otimes g^{\otimes 4} \otimes h \otimes k \otimes l\otimes u^{\otimes 4} \otimes v)(a_1,\dots,a_{13}),
\\
&
\big(\langle f,g, \langle h,k,l,u \rangle,v\rangle (\beta) \big)(a_1,\dots,a_{13})
=
\beta(f\otimes g \otimes \langle h,k,l,u\rangle \otimes v)(a_1,\dots,a_{13})
\\
&=\big( (m_1\circ_1\mu^\ast)\circ_9 \mu^\ast \big) (f\otimes g \otimes \langle h,k,l,u\rangle \otimes v)(a_1,\dots,a_{13})
\\
&=\big( (m_1\circ_1\mu^\ast)\circ_9 \mu^\ast \big)(f \otimes g \otimes h^{\otimes 4} \otimes k \otimes l\otimes u^{\otimes 4} \otimes v)(a_1,\dots,a_{13}),
\\
&
\big(\langle f,g,h, \langle k,l,u,v\rangle \rangle(\beta) \big)(a_1,\dots,a_{13})
=
\beta(f \otimes g \otimes h \otimes \langle k,l,u,v \rangle)(a_1,\dots,a_{13})
\\
&=
\big( (m_1\circ_1\mu^\ast)\circ_9 \mu^\ast \big) (f \otimes g \otimes h \otimes \langle k,l,u,v \rangle)(a_1,\dots,a_{13})
\\
&=
\big( (m_1\circ_1\mu^\ast)\circ_9 \mu^\ast \big)(f\otimes g \otimes h \otimes k^{\otimes 4} \otimes l\otimes u^{\otimes 4} \otimes v)(a_1,\dots,a_{13}).
\end{align*}
The second, third, and fourth results in this list may also be written respectively as follows using $m_2, m_3, m_4$ instead of $m_1$:
\begin{align*}
\langle\langle f,g,h,k \rangle,l,u,v \rangle ((m_2\circ_1\mu^\ast)\circ_9 \mu^\ast)(a_1,\dots,a_{13}),
\\
\langle\langle f,g,h,k \rangle,l,u,v \rangle ((m_3\circ_1\mu^\ast)\circ_9 \mu^\ast)(a_1,\dots,a_{13}),
\\
\langle\langle f,g,h,k \rangle,l,u,v \rangle ((m_4\circ_1\mu^\ast)\circ_9 \mu^\ast)(a_1,\dots,a_{13}).
\end{align*}
On the other hand, Lemma \ref{KoszuldualofTA4} shows that $m_1- m_2 + m_3 - m_4=0$ and hence
\begin{align*}
&
\Big(
 \langle \langle f, g, h, k \rangle, l, u, v \rangle (\beta)
- \langle f, \langle g, h, k, l \rangle, u, v \rangle (\beta)
\\
&\quad
+ \langle f, g, \langle h, k, l, u \rangle, v \rangle (\beta)
- \langle f, g, h, \langle k, l, u, v \rangle \rangle (\beta)
\Big)
(a_1,\dots,a_{13})
\\
&=
\beta \big( \langle\langle f,g,h,k \rangle,l,u,v \rangle \big)
\Big(\big((m_1\circ_1\mu^\ast)\circ_9 \mu^\ast\big) - \big((m_2\circ_1\mu^\ast)\circ_9 \mu^\ast\big)\\
&\quad
+ \big((m_3\circ_1\mu^\ast)\circ_9 \mu^\ast\big) - \big((m_4\circ_1\mu^\ast)\circ_9 \mu^\ast\big)\Big)(a_1,\dots,a_{13})
\\
&=
\beta \big( \langle\langle f,g,h,k \rangle,l,u,v \rangle \big)
\big(( (m_1-m_2+m_3-m_4) \circ_1\mu^\ast)\circ_9 \mu^\ast\big) (a_1,\dots,a_{13})
=0.
\end{align*}
This establishes the basis of the induction.

For the inductive step, we assume that
if $f$, $g$, $h$, $k$, $l$, $u$, $v$ are elements of degrees $n_1$, \dots, $n_7$ 
respectively such that $n_1 + \cdots + n_7 = r$,
then we have the relation
\begin{align*}
&
\big(\langle \langle f,g,h,k \rangle, l,u,v \rangle
- \langle f,\langle g,h,k,l \rangle,u,v\rangle
\\
&\quad
+ \langle f,g, \langle h,k,l,u\rangle,v\rangle
- \langle f,g,h, \langle k,l,u,v\rangle \big)(\alpha)=0,
\end{align*}
for all basis monomials $\alpha$ in $\mathcal{P}^{\text{!}}(3r+1)$.
We show that if the sum of the degrees of $f,\dots,v$ is $r+1$, then we have the relation
\begin{align*}
&
\big(\langle \langle f,g,h,k \rangle, l,u,v \rangle
- \langle f,\langle g,h,k,l \rangle,u,v\rangle
\\
&\quad
+ \langle f,g, \langle h,k,l,u\rangle,v\rangle
- \langle f,g,h, \langle k,l,u,v\rangle \big)(\beta)=0,
\end{align*}
for all basis monomials $\beta$ in $\mathcal{P}^{\text{!}}(3r+4)$.
Every such monomial has the form $\beta = \alpha \circ_i \mu^\ast$ where
$\alpha$ is a basis monomial for $\mathcal{P}^{\text{!}}(3r+1)$.
To simplify notation, we omit commas separating arguments and write
\[
\rho =
  \langle \langle f g h k \rangle l u v \rangle
- \langle f \langle g h k l \rangle u v \rangle
+ \langle f g \langle h k l u \rangle v \rangle
- \langle f g h \langle k l u v \rangle \rangle.
\]
We calculate as follows:
\begin{align*}
\rho( \beta )(a_1,\dots,a_{3r+4})
&=
\rho( \alpha \circ_i \mu^\ast )(a_1,\dots,a_{3r+4})
\\
&=
\rho( \alpha ) \big(a_1,\dots, \mu^\ast(a_i,a_{i+1}, a_{i+2}, a_{i+3}), \dots a_{3r+4}\big).
\end{align*}
By the inductive hypothesis, we have
$\rho( \alpha )( a_1,\dots,a_{3r+1} ) = 0$
for any $a_1, \dots, a_{3r+1}$.
Therefore
\[
\rho( \alpha ) \big( a_1,\dots, \mu^\ast(a_i,a_{i+1}, a_{i+2}, a_{i+3}), \dots a_{3r+4} \big ) = 0,
\]
for any $a_1, \dots, a_{3r+4}$.
To complete the proof, we note that it follows trivially from the the definition of
the cup product $\langle -,-,-,- \rangle$ that this quaternary operation on
$C_\mathcal{P}^\bullet(A,A)$ has (homological) degree two.
\end{proof}

\end{document}